\documentclass[12pt]{amsart}
\usepackage{amscd}
\usepackage{amsmath}

\usepackage{amssymb}

\newtheorem{prop}[equation]{Proposition}

\newtheorem{lemma}[equation]{Lemma}

\theoremstyle{definition}
\newtheorem{rem}[equation]{Remark}
\newtheorem{example}[equation]{Example}

\newcommand{\CH}{\mathop{\mathrm{CH}}\nolimits}

\newcommand{\pr}{\operatorname{\mathit{pr}}}

\newcommand{\inc}{\operatorname{\mathit{in}}}

\newcommand{\id}{\mathrm{id}}

\newcommand{\Spec}{\operatorname{Spec}}

\newcommand{\Prod}{\operatornamewithlimits{\textstyle\prod}}

\newcommand{\compose}{\circ}

\renewcommand{\phi}{\varphi}

\newcommand{\WR}{\mathcal{R}}
\newcommand{\WZ}{\mathcal{Z}}

\DeclareMathAlphabet{\cat}{OT1}{cmss}{m}{sl}

\title
[Pullback and Weil transfer]
{Pullback and Weil transfer \\
on Chow groups}


\author
{Nikita Karpenko}

\address
{Mathematical \& Statistical Sciences \\
University of Alberta \\
Edmonton
\\
CANADA}

\email
{karpenko@ualberta.ca}
\urladdr{www.ualberta.ca/~karpenko}

\author
{Guangzhao Zhu}

\address
{Mathematical \& Statistical Sciences \\
University of Alberta \\
Edmonton
\\
CANADA}

\email
{guangzha@ualberta.ca}

\date
{7 Apr 2025}

\thanks
{This work has been done during
the second named author's stay at the
Institut des Hautes Etudes Scientifiques.}

\begin{document}

\begin{abstract}
In the paper ``Weil transfer of algebraic cycles'', published by the second author in Indagationes Mathematicae about 25 years ago,
a {\em Weil transfer map} for Chow groups of smooth algebraic varieties has been constructed and its basic properties have been established.
The proof of commutativity with the pullback  homomorphisms given there used a variant of Moving Lemma suffering a lack of reference.
Here we are providing an alternative proof based on a more contemporary construction of the pullback via a deformation to the normal cone.
\end{abstract}

\maketitle



Let $F$ be a field.
By {\em $F$-variety}, we mean a quasi-projective $F$-scheme.

Let $L/F$ be a finite separable field extension and let $X$ be an $L$-variety.
We write $R(X)=R_{L/F}(X)$ for the $F$-variety given by the {\em Weil transfer}
(also called {\em Weil restriction})
of $X$ with respect to $L/F$,
see \cite[\S7.6]{MR1045822} or \cite[\S4]{MR1321819}.
In \cite{MR1809664}, a (non-additive) map of Chow groups
$$
\cat{R}\colon\CH(X)\to\CH(R(X))
$$
for smooth $X$ has been constructed,
called the {\em Weil transfer map}.
It satisfies the following property:
for any closed subvariety $Z\subset X$, the image under $\cat{R}$ of the class of $Z$ is the class of the closed subvariety
$R(Z)\subset R(X)$.

The map $\cat{R}$ is induced by the map
$$
\WR\colon\WZ(X)\to\WZ(R(X))
$$
of the groups of cycles defined as follows.

Let $E/F$ be a normal closure of the field extension $L/F$.
For any $F$-embedding
$$
\tau\colon L\hookrightarrow E,
$$
we define an $E$-variety $X_\tau$ as the base change of $X$ with respect to $\tau$:
$$
\begin{CD}
X_\tau @>>> X\\
@VVV @VVV\\
\Spec E @>{\tau}>> \Spec L
\end{CD}
$$
The canonical morphism of $L$-varieties $R(X)_L\to X$ induces
an isomorphism of the $E$-variety $R(X)_E$ with the product
$\Prod_{\tau} X_\tau$.
For any cycle $\alpha\in\WZ(X)$ and any $\tau$ as above, we write $\alpha_\tau$ for the pullback of $\alpha$ to $X_\tau$ via the morphism
$X_\tau\to X$.
We define $\WR(\alpha)$ as the cycle on $R(X)$ mapped
to the external product $\Prod_\tau\alpha_\tau$
under the base change homomorphism
$$
\WZ(R(X))\to \WZ(R(X)_E)=\WZ(\Prod_{\tau} X_\tau).
$$
The cycle $\WR(\alpha)$ exists and is uniquely determined by the above condition because
the base change homomorphism $\WZ(R(X))\to \WZ(R(X)_E)$ identifies $\WZ(R(X))$ with the group
$\WZ(R(X)_E)^G$ of $G$-invariant elements in $\WZ(R(X)_E)$, where $G$ is the Galois group of $E/F$.

Most of the properties of the map of Chow groups $\cat{R}$,
established in \cite{MR1809664}, are easy to verify because they hold ``on the level of cycles''.
For instance,

\begin{example}
\label{example}
For any two smooth $L$-varieties $X,Y$ and a {\em flat} morphism of schemes
$$
f\colon Y\to X,
$$
the morphism $R(f)\colon R(Y)\to R(X)$ is also flat, and
the square on the left
$$
\begin{CD}
\CH(Y) @<{f^*}<< \CH(X)\\
@V\cat{R}VV @VV\cat{R}V\\
\CH(R(Y)) @<R(f)^*<< \CH(R(X))
\end{CD}
\hspace{7em}
\begin{CD}
\WZ(Y) @<{f^*}<< \WZ(X)\\
@V{\WR}VV @VV{\WR}V\\
\WZ(R(Y)) @<R(f)^*<< \WZ(R(X))
\end{CD}
$$
commutes because by \cite[Proposition 3.5(flat pull-back)]{MR1809664} so does the square on the right.
\end{example}

The commutation with the general pullback homomorphism however
is more delicate because the latter is not defined on the level of cycles:

\begin{prop}[{\cite[Proposition 4.4(pull-back)]{MR1809664}}]
\label{prop}
For any morphism $f\colon Y\to X$ of smooth $L$-varieties $Y$ and $X$, the square
$$
\begin{CD}
\CH(Y) @<{f^*}<< \CH(X)\\
@V\cat{R}VV @VV\cat{R}V\\
\CH(R(Y)) @<R(f)^*<< \CH(R(X))
\end{CD}
$$
commutes.
\end{prop}

To prove Proposition \ref{prop}, a variant of Moving Lemma suffering a lack of reference
(see \cite[Appendix A]{MR3617981})
has been used in \cite{MR1809664}.\footnote{We thank Stefan Gille for pointing this out.}
Here we are providing an alternative proof based on the ``modern'' definition of the pullback via the deformation to the normal cone
homomorphism.
More precisely, we will use the modified approach developed by Markus Rost in \cite{MR1418952} with its detailed exposition given
in \cite{EKM}, which is simpler than the original approach of \cite{fulton}.

First of all,
the homomorphism $f^*$ is defined (see \cite[(55.15)]{EKM}) as the composition $\inc^*\compose\pr^*$,
where $$\pr\colon Y\times X\to X$$ is the projection and
$$
\inc:=(\id_Y,f)\colon Y\to Y\times X.
$$
Taking into account the identification $R(Y\times X)=R(Y)\times R(X)$, the morphism $R(\pr)$ is the projection
$R(Y)\times R(X)\to R(X)$ whereas $R(\inc)=(\id_{R(Y)},R(f))$.

The morphisms $\pr$ and $R(\pr)$ are flat so that the pullbacks $\pr^*$ and $R(\pr)^*$ are defined on the level of cycles;
the squares
$$
\begin{CD}
\WZ(Y\times X) @<{\pr^*}<< \WZ(X)\\
@V{\WR}VV @VV{\WR}V\\
\WZ(R(Y\times X)) @<R(\pr)^*<< \WZ(R(X))
\end{CD}
\hspace{3em}
\text{ and }
\hspace{3em}
\begin{CD}
\CH(Y\times X) @<{\pr^*}<< \CH(X)\\
@V\cat{R}VV @VV\cat{R}V\\
\CH(R(Y\times X)) @<R(\pr)^*<< \CH(R(X))
\end{CD}
$$
commute by Example \ref{example}.

The morphism $\inc$ is a regular closed embedding.
In fact, any closed embedding of smooth varieties is regular (see \cite[Proposition 104.16]{EKM}).
Since the Weil transfer functor preserves smoothness and closed embeddings, we reduced the proof
of Proposition \ref{prop} to the case where $f$ is a closed embedding.

Once we assume $f$ is a closed embedding,
the pullback homomorphism $f^*$ is the Gysin homomorphism
defined (see \cite[55.A]{EKM}) as the composition $(p_f^*)^{-1}\compose\sigma_f$,
where
$$
\sigma_f\colon\CH(X)\to\CH(N_f)
$$
is the deformation homomorphism and
$p_f\colon N_f\to Y$ is the vector bundle over $Y$ given by the normal cone of $f$.
By Homotopy Invariance of Chow groups \cite[Theorem 52.13]{EKM},
the flat pullback $p_f^*\colon\CH(Y)\to\CH(N_f)$
is an isomorphism.

We claim that the normal bundle $N_{R(f)}$
of the closed embedding $R(f)\colon R(Y)\to R(X)$
is given by the Weil transfer of $N_f$:
\begin{equation}
\label{NRf}
N_{R(f)}=R(N_f).
\end{equation}
To see it,
note that
by \cite[(4.2.3)]{MR1321819},
$R(X)_L$ can be obtained as the Weil transfer of $X_K$ with respect to
the \'etale $L$-algebra $K:=L\otimes_F L$.
This $L$-algebra splits off $L$ as a direct factor:
$K=L\times K'$ for certain \'etale $L$-algebra $K'$.
Thus, by \cite[(4.2.6)]{MR1321819}, $R(X)_L=X\times R'(X)$, where $R'(X)$ is the Weil transfer of $X_{K'}$.
Note that the canonical morphism $R(X)_L\to X$ is given by the projection $X\times R'(X)\to X$.
Recall that the induced morphism of $E$-varieties $R(X)_E\to\Prod_\tau X_\tau$, where $\tau$ runs
over the $F$-embeddings $L\hookrightarrow E$, is an isomorphism.

It follows that the closed embedding $R(f)_L\colon R(Y)_L\to R(X)_L$ is the direct product
$$f\times R'(f)\colon Y\times R'(Y)\to X\times R'(X)$$
of the closed embeddings
$Y\to X$ and $R'(Y)\to R'(X)$, and so
$$
(N_{R(f)})_L=N_{R(f)_L}=N_f\times N_{R'(f)}
$$
by naturality of the normal cone \cite[Proposition 104.23]{EKM} and its commutation with direct products
 \cite[Proposition 104.7]{EKM}.
The first projection $(N_{R(f)})_L\to N_f$ induces an isomorphism
$(N_{R(f)})_E\to\Prod_\tau (N_f)_\tau$ proving claim (\ref{NRf}),
cf.\! \cite[\S2.8]{MR0181643}.

Recall that the morphism $$p_f\colon N_f\to Y$$ is flat and the pullback $$p_f^*\colon\CH(Y)\to\CH(N_f)$$ is an isomorphism.
Similarly, the morphism $$R(p_f)\colon R(N_f)\to R(Y)$$ is flat and the pullback $$R(p_f)^*\colon\CH(R(Y))\to\CH(R(N_f))$$ is an isomorphism.
We already know (due to the fact that $p_f^*$ and $R(p_f)^*$ are defined on the level of cycles) that the square
$$
\begin{CD}
\CH(Y) @>{p_f^*}>> \CH(N_f)\\
@V\cat{R}VV @VV\cat{R}V\\
\CH(R(Y)) @>R(p_f)^*>> \CH(R(N_f))
\end{CD}
$$
commutes.
It follows that the square
$$
\begin{CD}
\CH(Y) @<{(p_f^*)^{-1}}<< \CH(N_f)\\
@V\cat{R}VV @VV\cat{R}V\\
\CH(R(Y)) @<{(R(p_f)^*)^{-1}}<< \CH(R(N_f))
\end{CD}
$$
with the inverses of
$p_f^*$ and $R(p_f)^*$ (not defined on the level of cycles anymore) commutes as well.

As per \cite[\S51]{EKM},
the deformation homomorphism $\sigma_f$ is also defined on the level of cycles.
Because of that, our last remaining step in the proof of Proposition \ref{prop} is not difficult to perform:

\begin{lemma}
The squares
$$
\begin{CD}
\CH(N_f) @<{\sigma_f}<< \CH(X)\\
@V\cat{R}VV @VV\cat{R}V\\
\CH(R(N_f)) @<{\sigma_{R(f)}}<< \CH(R(X))
\end{CD}
\hspace{3em}
\text{ and }
\hspace{3em}
\begin{CD}
\WZ(N_f) @<{\sigma_f}<< \WZ(X)\\
@V{\WR}VV @VV{\WR}V\\
\WZ(R(N_f)) @<{\sigma_{R(f)}}<< \WZ(R(X))
\end{CD}
$$
commute.
\end{lemma}

\begin{proof}
We only need to treat the second square: its commutativity implies commutativity of the first one.

Since the base change homomorphism $\WZ(R(N_f))\to\WZ(R(N_f)_E)$ is injective,
we may perform the base change $E/F$ in the lower line of the $\WZ$-square.
Then it becomes
$$
\begin{CD}
\WZ(N_f) @<{\sigma_f}<< \WZ(X)\\
@V{\beta\mapsto\prod_\tau\beta_\tau}VV @VV{\alpha\mapsto\prod_\tau\alpha_\tau}V\\
\WZ(\prod_\tau N_{f_\tau}) @<{\prod_\tau\sigma_{f_\tau}}<< \WZ(\prod_\tau X_\tau)
\end{CD}
$$
and is commutative because
$\sigma_{f_\tau}(\alpha_\tau)=\sigma_f(\alpha)_\tau$
for any cycle $\alpha\in\WZ(X)$ and $F$-embedding $\tau\colon L\hookrightarrow E$ according to \cite[Proposition 51.5]{EKM}.
\end{proof}

\begin{rem}
By \cite[Proposition 4.4(interior product)]{MR1809664}, the Weil transfer map $$\cat{R}\colon\CH(X)\to\CH(R(X))$$
is {\em multiplicative}:
\begin{equation}
\label{formula}
\cat{R}(\alpha\cdot\beta)=\cat{R}(\alpha)\cdot\cat{R}(\beta)
\end{equation}
for any $\alpha,\beta\in\CH(X)$.
This follows from \cite[Proposition 4.4(exterior product)]{MR1809664} giving the similar formula
$\cat{R}(\alpha\times\beta)=\cat{R}(\alpha)\times\cat{R}(\beta)$
for the external product, the formula $\alpha\cdot\beta=\delta^*(\alpha\times\beta)$
(see \cite[(56.1)]{EKM})
expressing the internal product as
the pullback of the external one with respect to the diagonal morphism $\delta\colon X\to X\times X$, and
Proposition \ref{prop}.
Therefore the new proof of  Proposition \ref{prop}, given here, provides a new proof for (\ref{formula}) as well.
\end{rem}

\medskip
\noindent
{\sc Acknowledgements.}
We thank Stefan Gille for pointing out the problem and checking through the solution.


\end{document}